\documentclass[12pt, a4paper,oneside,centertags,reqno]{amsart}

\usepackage{amscd}
\usepackage{amsfonts}
\usepackage{amsmath}
\usepackage{amssymb}
\usepackage{amsthm}
\usepackage[usenames, dvipsnames]{color}
\usepackage{endnotes}
\usepackage{enumerate}
\usepackage{epsfig}
\usepackage{eucal}
\usepackage{graphpap}
\usepackage{hyperref}
\usepackage{latexsym}
\usepackage{mathrsfs}
\usepackage{pdfsync}

 \usepackage [T1] {fontenc}
 \usepackage [utf8] {inputenc}
 \usepackage {url}
 \usepackage{graphicx}

\font\gotxii=eufm10 at 12pt

\font\posebni=msam10
\font\tretamajka=cmbsy10 at 11pt

\input cyracc.def

\newcommand{\e}{\varepsilon}
\newcommand{\f}{\varphi}

\newcommand\bS{\mathbf{S}}

\newcommand{\nor}[1]{|\hskip -0.6pt | #1 |\hskip -0.6pt |}

\newcommand\norm[2]{{\left\Vert{#1}\right\Vert_{#2}}}
\newcommand{\mn}[2]{\{ #1\, ;\, #2 \}}
\newcommand{\Mn}[2]{\left\{ #1\, ;\, #2 \right\}}
\newcommand{\sk}[2]{\left\langle #1 , #2\right\rangle}

\renewcommand{\div}[0]{{\rm div}}
\newcommand{\sign}{\mathop{\rm sign}\nolimits}

\newcommand{\cA}[0]{\mathcal A}

\newcommand{\cD}[0]{\mathcal D}
\newcommand{\cE}[0]{\mathcal E}

\newcommand{\cH}[0]{\mathcal H}
\newcommand{\cI}[0]{\mathcal I}

\newcommand{\cL}{\mathcal L}
 
\newcommand{\cO}[0]{\mathcal O}

\newcommand{\cQ}{\mathcal Q}

\newcommand{\cV}[0]{\mathcal V} 
\newcommand{\cW}[0]{\mathcal W} 
\newcommand{\bP}{{\mathbf P}}

\renewcommand{\a}[0]{\text{\gotxii a}}
\newcommand{\gota}[0]{\text{\gotxii a}}
\newcommand{\gotb}[0]{\text{\gotxii b}}

\newcommand{\leqsim}[0]{\,\text{\posebni \char46}\,}
\newcommand{\geqsim}[0]{\,\text{\posebni \char38}\,}

\newcommand{\oL}{{\mathscr L}}
\newcommand{\oN}{{\mathscr N}}
\newcommand{\oP}{{\mathscr P}}

\newcommand{\oV}{{\mathscr V}}
\newcommand{\oW}{{\mathscr W}}

\renewcommand{\geq}[0]{\geqslant}
\renewcommand{\leq}[0]{\leqslant}
\renewcommand{\Re}[0]{{\rm Re}\,}
\renewcommand{\Im}[0]{{\rm Im}\,}

\newcommand{\N}[0]{{\mathbb N}}

\newcommand{\R}[0]{\mathbb{R}}

\definecolor{mojaboja}{RGB}{130,60,40}
\definecolor{mojabojaa}{RGB}{100,200,100}

\newcommand{\pd}[0]{\partial}

\renewcommand{\theta}[0]{\vartheta}
\renewcommand\mod[1]{\left\vert{#1}\right\vert}

\newtheorem{thm}{Theorem}
\newtheorem{predefinition}[thm]{Definition}
\newenvironment{definition}
{\begin{predefinition}\rm}{\end{predefinition}}
\newtheorem{lemma}[thm]{Lemma}
\newtheorem{prop}[thm]{Proposition}
\newtheorem*{prop*}{Proposition}
\newtheorem{cor}[thm]{Corollary}

\newtheorem{preexample}[thm]{Example}  
\newtheorem{preremark}[thm]{Remark}  \newenvironment{remark}
{\begin{preremark}\rm}{\end{preremark}}
\newtheorem{prenotation}[thm]{Notation}  \newenvironment{notation}
{\begin{prenotation}\rm}{\end{prenotation}}

\numberwithin{equation}{section}
\numberwithin{thm}{section}

\begin{document}

\title[Bilinear embedding for Schr\"odinger-type operators]
{Bilinear embedding for Schr\"odinger-type operators with complex coefficients}


\author[Carbonaro]{Andrea Carbonaro}
\author[Dragi\v{c}evi\'c]{Oliver Dragi\v{c}evi\'c}

\address{Andrea Carbonaro\\ Universit\`a degli Studi di Genova\\ Dipartimento di Matematica\\ Via Dodecaneso\\ 35 16146 Genova\\ Italy}
\email{carbonaro@dima.unige.it}

\address{Oliver Dragičević, Department of Mathematics, Faculty of Mathematics and Physics, University of Ljubljana,  Jadranska ulica 19, SI-1000 Ljubljana, Slovenia, and Institute of Mathematics, Physics and Mechanics, Jadranska ulica 19, SI-1000 Ljubljana, Slovenia}
\email{oliver.dragicevic@fmf.uni-lj.si}

\begin{abstract}
We prove a variant of the so-called bilinear embedding theorem for operators in divergence form 
with complex coefficients and 
with nonnegative locally integrable potentials, 
subject to mixed boundary conditions, and 
acting on arbitrary open subsets of $\R^{d}$. 
\end{abstract}

\maketitle

\section{Introduction and the statement of the main result}
\label{s: intro}

Let $\Omega\subseteq\R^{d}$ be an arbitrary open set. 
Denote by $\cA(\Omega)$ the family of all complex {\it uniformly strictly accretive} (also called {\it elliptic}) $n\times n$ matrix functions on $\Omega$ with $L^{\infty}$ coefficients. 
That is, the set of all measurable $A:\Omega\rightarrow{\mathbb C}^{d\times d}$ for which 
there exist $\lambda,\Lambda>0$ such that for almost all $x\in\Omega$ we have
\begin{eqnarray}
\label{eq: ellipticity}
\Re\sk{A(x)\xi}{\xi}
&\hskip -19pt\geq \lambda|\xi|^2\,,
&\hskip 20pt\forall\xi\in{\mathbb C}^{d};
\\
\label{eq: boundedness}
\mod{\sk{A(x)\xi}{\eta}}
&\hskip-6pt\leq \Lambda \mod{\xi}\mod{\eta}\,,
&\hskip 20pt\forall\xi,\eta\in{\mathbb C}^{d}.
\end{eqnarray}
Elements of $\cA(\Omega)$ will also more simply be referred to as {\it accretive} 
or {\it elliptic matrices}.
For any $A\in\cA(\Omega)$ denote by
$\lambda(A)$
the largest admissible $\lambda$ in \eqref{eq: ellipticity} and by 
$\Lambda(A)$
the smallest $\Lambda$ in \eqref{eq: boundedness}.

Denote by $H^{1}_{0}(\Omega)$ the closure of $C_c^\infty(\Omega)$ in the Sobolev space $H^1(\Omega)=W^{1,2}(\Omega)$. 
Let $\mathscr{V}$ be a closed subspace of $H^{1}(\Omega)$ containing $H^{1}_{0}(\Omega)$, that is
\begin{equation}
\label{eq: Paradis Sicilienne}
H^{1}_{0}(\Omega)
\subseteq\oV
\subseteq
H^{1}(\Omega).
\end{equation}
Recall that $H^{1}_{0}(\R^{d})=H^{1}(\R^{d})$; see \cite[Corollary~3.19]{Adams} for a reference.

Furthermore, let $V\in L_{loc}^1(\Omega)$ be a nonnegative function. 
We define the operator, formally denoted by $Lu=-\div(A\nabla u)+Vu$, in the standard manner via  sesquilinear forms, see e.g. \cite[Sections 4.1 and 4.7]{O}.
Before proceeding we state 
that all the integrals in this paper will be taken with respect to the Lebesgue measure. As the ambient space we will always take $\cH=L^2(\Omega)$.

Let the form $\a=\a_{A,V}=\a_{A,V,\oV}$ be given by its domain
\begin{equation}
\label{eq: mongolska}
\aligned
\cD(\a)
&=\Mn{u\in\oV}{\int_\Omega V|u|^2<\infty}\\
\endaligned
\end{equation}
and, for $u,v\in\cD(\a)$, the formula
\begin{equation}
\label{eq: govedina}
\a(u,v):=\int_\Omega\left(\sk{A\nabla u}{\nabla v}_{{\mathbb C}^n} 
+ Vu\bar v\right).
\end{equation}

We define $L=L_{A,V}=L_{A,V,\oV}$ to be the 
unbounded, densely defined, closed operator on $L^2(\Omega)$, associated with $\a_{A,V}$. See \cite[Section 1.2.3]{O} for information about this construction. 
So we have
\begin{equation}
\label{eq: rakar230319}
\int_\Omega\sk{Lu}{v}_{{\mathbb C}^n}=
\int_\Omega\left(\sk{A\nabla u}{\nabla v}_{{\mathbb C}^n} +
Vu\bar v\right)
\hskip 20pt
\forall u\in\cD(L),
v\in\cD(\a).
\end{equation}
In accordance with \cite[1.8]{D} we call $L$ a {\it generalized Schr\"odinger operator} and $V$ its {\it potential}.

Given $\theta\in (0,\pi)$ define the {\it (open) sector of angle $\theta$} by
\begin{equation*}
\label{autobus karlovac-posedarje}
\bS_{\theta}=\{z\in{\mathbb C}\setminus\{0\}\,;\ |\arg z|<\theta\}.
\end{equation*}
Also set $\bS_{0}=(0,\infty)$. 
The basic properties of $\a$ are 
recalled
in the following result. 
Let $\nu(A)$ \label{Borodin Quartet} be the opening angle of the smallest sector 
whose closure 
contains the {\it numerical range} of $A$; see 
\cite[(2.9)]{CD-DivForm} for the definition. 
It follows from the definition of accretivity that 
$A\in\cA(\Omega)$ implies $0\leq\nu(A)\leq\arccos(\lambda/\Lambda) <\pi/2$. 
Other notions appearing in the next statement can be found in \cite{O}.

\begin{thm}
\label{t: majke cigan}
For every $\phi\in\R$ such that $|\phi|<\pi/2-\nu(A)$, the form $e^{i\phi}\a$ is densely defined, sectorial and closed.
\end{thm}
Sectorial forms are automatically accretive and continuous; see e.g. \cite[Proposition 1.8]{O}.
Therefore, by \cite[Proposition~1.27 and Theorem~1.54]{O}, the operator $-L$ generates on $L^2(\Omega)$ a strongly continuous semigroup of operators
$$
T_{t}=T^{A,V}_{t}=T^{A,V,\oV}_{t}=\exp(-tL),\quad t>0,
$$
which is holomorphic and contractive in a sector of positive angle. Hence $T_{t}$ maps  $L^{2}(\Omega)$ into $\cD(L)\subseteq\oV$ 
\cite[Theorem II.4.6]{EN}, 
and since $\oV\subseteq H^1(\Omega)$, 
the spatial gradient $\nabla  T_{t}f$ is well defined. By \cite[p. 72]{Stein}, 
given $f\in L^{2}(\Omega)$ we can redefine each $T_{t}f$ on a set of measure zero, in such a manner that for almost every $x\in\Omega$ the function $t\mapsto T_{t}f(x)$ is real-analytic on $(0,\infty)$.

\subsection{Special classes of boundary conditions}
\label{s: Karlstejn}

Here we describe certain classes of domains $\oV$ that, on top of  
\eqref{eq:  Paradis Sicilienne},
satisfy additional conditions which will be assumed throughout the rest of this paper.

\smallskip

We say that the space $\oV\subset H^{1}(\Omega)$ is {\it invariant} under: 
\begin{itemize}
\item
the function $p:{\mathbb C}\rightarrow{\mathbb C}$, if $u\in\oV$ implies $p(u):=p\circ u\in\oV$;
\item 
the family $\oP$ of functions ${\mathbb C}\rightarrow{\mathbb C}$, if it is invariant under all $p\in\oP$.
\end{itemize}

Define function $P:{\mathbb C}\rightarrow{\mathbb C}$ by 
\begin{equation}
\label{eq: P}
P(\zeta)=
\left\{
\begin{array}{lll}
\zeta & ; & |\zeta|\leq1\\
\zeta/|\zeta| & ; & |\zeta|\geq1.
\end{array}
\right.
\end{equation}
Thus $P(\zeta)=\min\{1,|\zeta|\}\sign\zeta$, where 
$\sign$ is defined as in \cite[(2.2)]{O}: 
$$
\sign\zeta:=\left\{
\begin{array}{ccl}
\zeta/|\zeta| & ; & \zeta\ne0\\
0 & ; & \zeta=0.
\end{array}
\right.
$$
Let $\mathscr{V}$ be a closed subspace of $H^{1}(\Omega)$ containing $H^{1}_{0}(\Omega)$ and such that
\begin{equation}
\label{eq: invariance}
\oV \text{ is invariant under the function } P.
\end{equation}
It is well known, see Ouhabaz \cite[Proposition 4.11]{O}, that \eqref{eq: invariance} is satisfied in these notable cases which will feature in our {\it bilinear embedding} (Theorem \ref{t: bilinemb}):
\begin{enumerate}[(a)]
\label{eq: Hyundai}
\item
$\oV=H^{1}_{0}(\Omega)$
\item
$\oV=H^{1}(\Omega)$
\item
$\oV$ is the closure in $H^{1}(\Omega)$ of $\Mn{u|_\Omega}{u\in C_c^\infty(\R^{d}\backslash\Gamma)}$, where $\Gamma$ is a (possibly empty) closed subset of $\pd\Omega$.
\end{enumerate}

When $\oV$ falls into any of the special cases (a)-(c) from
Section \ref{s: Karlstejn}, 
we say that $L=L_{A,V,\oV}$ is subject to (a) {\it Dirichlet}, (b) {\it Neumann} or (c) {\it mixed boundary conditions}.

\subsection{The $p$-ellipticity condition} 
\label{s: merlin}
The concept of $p$-ellipticity was introduced by the present authors in \cite{CD-DivForm} as follows.

Given $A\in\cA(\Omega)$ and $p\in (1,\infty)$, we say that $A$ is {\it $p$-elliptic} if
$\Delta_{p}(A)>0$, where
\begin{equation}
\label{eq: kabuto}
\Delta_{p}(A):=
\underset{x\in\Omega}{{\rm ess}\inf}
\min_{\substack{\xi\in{\mathbb C}^{d}\\ |\xi|=1}}
\Re\sk{A(x)\xi}{\xi+|1-2/p|\bar\xi}_{{\mathbb C}^{d}}.
\end{equation}
Equivalently, $A$ is $p$-elliptic if
there exists $C=C(A,p)>0$ such that for a.e. $x\in\Omega$,
\begin{equation}
\label{eq: Sparky 21}
\Re\sk{A(x)\xi}{\xi+|1-2/p|\bar\xi}_{{\mathbb C}^{d}}
\geq C |\xi|^2\,,
\hskip30pt
 \forall\xi\in{\mathbb C}^{d}.
\end{equation}
It follows straight from \eqref{eq: kabuto}
that $\Delta_{p}$ is invariant under conjugation of $p$, meaning that $\Delta_{p}(A)=\Delta_{q}(A)$ when $1/p+1/q=1$.
Furthermore, note that $\Delta_2(A)=\lambda(A)$, so $p$-ellipticity generalizes the notion of classical ellipticity.

Denote by $\cA_p(\Omega)$ the class of all $p$-elliptic matrix functions on $\Omega$. It is known, see \cite{CD-DivForm}, that 
$\mn{\cA_p(\Omega)}{p\in[2,\infty)}$
is a decreasing chain of matrix classes such that
$$
\begin{array}{rcc}
\{\text{elliptic matrices on }\Omega\} & = & \cA_2(\Omega)\\
\{\text{real elliptic matrices on }\Omega\} & = & 
{\displaystyle\bigcap_{p\in[2,\infty)}\cA_p(\Omega)}\,.
\end{array}
$$
Since we will be dealing with pairs of matrices, it is useful to introduce further notation, as in \cite{CD-DivForm, CD-Mixed}:
\begin{equation}
\aligned
\Delta_p(A,B)&=\min\{\Delta_p(A), \Delta_p(B)\}\\
\lambda(A,B)&=\min\{\lambda(A), \lambda(B)\}\\
\Lambda(A,B)&=\max\{\Lambda(A), \Lambda(B)\}.
\endaligned
\end{equation}

While the present authors were preparing \cite{CD-DivForm}, M. Dindo\v{s} and J. Pipher
were working on their own article \cite{DiPi}.
They found a sharp condition, see \cite[(1.3)]{DiPi}, which implies reverse H\"older inequalities for weak solutions of elliptic operators in divergence form with complex coefficients. It turned out that the condition of theirs, devised
independently of \cite{CD-DivForm}, namely, as a strengthening of \cite[(2.25)]{CM}, was 
exactly equivalent to \eqref{eq: Sparky 21}.
The same authors have since then been successfully continuing their line of exploration of $p$-ellipticity in PDEs; see their recent papers \cite{DiPi20, DiLiPi20}.

\smallskip
The notion of $p$-ellipticity emerged in \cite{CD-DivForm} after several years of gradually distilling the {\it Bellman-function-heat-flow-method} (see Section \ref{s: Heat flow}), initiated in \cite{PV, NV},  through \cite{DV-Kato, DV-Sch, DV, CD-mult, CD-OU, CD-Riesz}.
More information about the genesis of $p$-ellipticity can be found in \cite{CD-DivForm, Trilinear}.

\subsection{Semigroup properties on $L^p$}

The first result of ours is Theorem \ref{t: wabaus}. It generalizes the implication ${\mathit(a)}\Rightarrow{\mathit(b)}$ of \cite[Theorem 1.3]{CD-DivForm}. See also \cite[Proposition 1]{CD-Mixed}, where it was proven in special cases (a)-(c) from Section \ref{s: Karlstejn}, and $\phi=0$, $V=0$.  
The proof of Theorem \ref{t: wabaus} is a modification of the one from \cite{CD-DivForm}, the main difference being that instead of \cite[Theorem 4.7]{O} we now use a more 
general result
\cite[Theorem 4.31]{O}. In all of those cases, we build on a criterion by Nittka (Theorem \ref{t: Nittka theorem}).
Assuming again that $\phi=0$, $V=0$ and $\oV$ is one of the special cases (a)-(c) from Section \ref{s: Karlstejn}, a proof of Theorem \ref{t: wabaus} different from the one above, yet still resting on Nittka's theorem, was recently found by Egert \cite[Proposition 13]{Egert2018}. 
Compare also with theorems by ter Elst et al. \cite{tELSV, Elst2019}.

\begin{thm}
\label{t: wabaus}
Choose 
$p>1$, 
$A\in\cA(\Omega)$, 
$\phi\in\R$ such that $|\phi|<\pi/2-
\nu(A)$
and $\Delta_p(e^{i\phi}A)\geq0$, 
and a nonnegative $V\in L_{\rm loc}^1(\Omega)$. 
Then,
for every $\oV$ satisfying \eqref{eq: invariance}, 
$$
\big(e^{-te^{i\phi}L_{A,V,\oV}}\big)_{t>0}
$$
extends to a contractive semigroup on $L^p(\Omega)$. 
\end{thm}

The next corollary extends \cite[Lemma 17]{CD-Mixed}.

\begin{cor}
\label{c: Marof}
Choose $p>1$, $A\in\cA_p(\Omega)$ and a nonnegative $V\in L_{\rm loc}^1(\Omega)$. 
Let $\oV$ satisfy \eqref{eq: invariance}.
Then there exists $\theta=\theta(p,A)>0$ such that if $|1-2/r|\leq|1-2/p|$ then 
$\Mn{T_z^{A,V}}{z\in\bS_{\theta}}$ is holomorphic and contractive in $L^r(\Omega)$.
\end{cor}

The proofs of these results will be given in Section \ref{s: Veco Holjevac}.

\subsection{Main result: the bilinear embedding theorem for pairs of complex $p$-elliptic operators with mixed boundary conditions}

In this section we assume boundary conditions that are less general than those from our contractivity result (Theorem \ref{t: wabaus}). Namely, we take pairs $\oV,\oW$ which are of the form (a)-(c) from Section \ref{s: Karlstejn}. 
We needed this restriction in order to tackle technical issues which arise in the proof of this note's main result, the {\it dimension--free bilinear embedding theorem} which we formulate next. 

\begin{thm}
\label{t: bilinemb}
Choose $p>1$. Let $q$ be its conjugate exponent, i.e., $1/p+1/q=1$. Suppose that $A,B\in\cA_p(\Omega)>0$. Let $V,W\in L_{\rm loc}^1(\Omega)$ be nonnegative. Assume that the operators $L_{A,V,\oV}$ and $L_{B,W,\oW}$ are subject to Dirichlet, Neumann or mixed boundary conditions, cf. {\rm (a)-(c)} from Section \ref{s: Karlstejn}.

There exists $C>0$ such that for any $f,g\in (L^p\cap L^q)(\Omega)$ we have
$$
\int_0^\infty\int_{\Omega}
\sqrt{\big|\nabla T^{A,V,\oV}_{t}f\big|^2+V\big|T^{A,V,\oV}_{t}f\big|^2}
\sqrt{\big|\nabla T^{B,W,\oW}_{t}g\big|^2+W\big|T^{B,W,\oW}_{t}g\big|^2}
\leq C\nor{f}_p\nor{g}_q.
$$
We may choose $C>0$ which depends on 
$p,A,B$,
but not on the dimension $d$ nor the potentials $V,W$.
\end{thm}

This result incorporates several earlier theorems as special cases, including:
\begin{itemize}
\item
$V=W$, $\Omega=\R^{d}$, $A=B$ equal and real \cite[Theorem 1]{DV-Kato}

\item
$V=W=0$, $\Omega=\R^{d}$ \cite[Theorem 1.1]{CD-DivForm}

\item
$V=W=0$ \cite[Theorem 2]{CD-Mixed}.
\end{itemize}
See also \cite[Theorem 1]{DV-Sch} for a variant with $A=B=I$, $V=W$, $\Omega=\R^{d}$ and involving the semigroup generated by the {\sl square root} of the operator $L$. 
This variant also bore consequences in the shape of dimension-free $L^p$ estimates of Riesz transforms associated with the harmonic oscillator (Hermite operator), which are optimal with respect to $p$; see \cite[Corollary 1]{DV-Sch}. 
The difficulty that we encounter in this paper is the generality of the setting in terms of domains ($\Omega$), matrices ($A,B$) and potentials ($V,W$), which requires significantly more effort in order to complete the proof.

Various types of bilinear embeddings have been proven in the last 20 years, often admitting important consequences, such as Riesz transform estimates and optimal holomorphic functional calculus. The present authors' efforts aimed at proving bilinear embedding as in \cite[Theorem 1.1]{CD-DivForm} eventually gave rise to the concept of $p$-ellipticity summarized in Section \ref{s: merlin}.
See \cite{Trilinear} and the above references for more historical background and motivation. We finally remark that  $p$-ellipticity is the sharp condition for {\sl dimension-free} bilinear embeddings; see \cite[Section 1.4]{CD-DivForm} for a precise statement.

\section{Proof of Theorem \ref{t: wabaus}}
\label{s: Veco Holjevac}

We first recall the notion of {\it $L^{p}$-dissipativity} of sesquilinear forms. It was introduced by Cialdea and Maz'ya in \cite[Definition~1]{CM} for the case of forms defined on $C_c^1(\Omega)$ and associated with complex matrices. In \cite[Definition 7.1]{CD-DivForm} the present authors extended their definition as follows.

\begin{definition}
\label{d: linna}
Let $X$ be a measure space, $\gotb$ a sesquilinear form defined on the domain $\cD(\gotb)\subset L^2(X)$ and $1<p<\infty$. Denote 
$$
\cD_{p}(\gotb):=\mn{u\in\cD(\gotb)}{|u|^{p-2}u\in\cD(\gotb)}. 
$$
We say that $\gotb$ is {\it $L^{p}$-dissipative} if 
$$
\Re\gotb\left(u,|u|^{p-2}u\right)\geq0
\hskip40pt
\forall\,u\in\cD_{p}(\gotb).
$$
\end{definition}

The following theorem is due to Nittka \cite[Theorem 4.1]{Nittka}. 
We remark that Nittka formulated his result for sectorial forms, but it seems that sectoriality is not needed for our version of his result, since it is not needed for Ouhabaz's criterion \cite[Theorem 2.2]{O} on which Nittka's own criterion is based. Of course, the forms we are dealing with in this paper are all sectorial anyway.

\begin{thm}
\label{t: Nittka theorem} 
Let $(\Omega,\mu)$ be a measure space. 
Suppose that the sesquilinear form $\a$ on $L^2=L^2(\Omega,\mu)$ is densely defined, accretive, 
continuous and closed. Let $\cL$ be the operator associated with $\a$ in the sense of \cite[Section 1.2.3]{O}.

Take $p\in(1,\infty)$ and define $B^p:=\mn{u\in L^2\cap L^p}{\nor{u}_p\leq 1}$. Let $\bP_{B^p}$ be the orthogonal projection $L^2\rightarrow B^p$. 
Then the following assertions are equivalent: 
\begin{itemize}
\item
\label{eq: goro}
$\nor{\exp(-t\cL)f}_p\leq\nor{f}_p$ for all $f\in L^2\cap L^p$ and all $t\geq0$;
\item
\label{eq: masamune}
$\cD(\a)$ is invariant under $\bP_{B^p}$ and 
$\text{\gotxii a}$ is {\it $L^p$-dissipative}.
\end{itemize}
\end{thm}

Define for $p>1$ the operator $\cI_p:{\mathbb C}^n\rightarrow{\mathbb C}^n$ by 
\begin{equation}
\label{eq: J_p}
\cI_p\xi=\xi+(1-2/p)\bar\xi.
\end{equation}
Observe that $\cI_p$ appears in 
\eqref{eq: kabuto} and \eqref{eq: Sparky 21}.

The formula below follows from the chain rule for Sobolev functions. It already appeared in \cite[(7.5)]{CD-DivForm}, where it was proven for $f,|f|^{p-2}f\in H_{0}^{1}(\Omega)$. Basically the same proof works if $f,|f|^{p-2}f\in H^{1}(\Omega)$.
Here we restate it for the reader's convenience. It will be used in the proof of Theorem \ref{t: wabaus}.

\begin{lemma}
\label{l: chain revisited}
Suppose that $p>1$ and that $f$ and $|f|^{p-2}f$ belong to $H^1(\Omega)$. Then 
\begin{equation}
\label{eq: kamilica}
\nabla\left(|f|^{p-2}f\right)=\frac{p}{2}|f|^{p-2}\sign f\cdot\cI_p\left(\sign\bar f\cdot\nabla f\right).
\end{equation}
\end{lemma}
Consequently, for $f,p$ as in Lemma \eqref{l: chain revisited} and any $B\in\cA(\Omega)$ we have
\begin{equation}
\label{eq: bazilika}
\sk{B\nabla f}{\nabla\left(|f|^{p-2}f\right)}_{{\mathbb C}^{d}}
=
\frac{p}{2}|f|^{p-2}\sk{B\left(\sign\bar f\cdot\nabla f\right)}{\cI_p\left(\sign\bar f\cdot\nabla f\right)}_{{\mathbb C}^{d}}.
\end{equation}
Note the symmetric structure of the inner product above, which is expressed in appearance of $(\sign\bar f)\nabla f$ in both factors. 

Taking real parts and recalling \eqref{eq: kabuto} and \eqref{eq: J_p} we conclude
\begin{equation}
\label{eq: knjiga}
\Re\sk{B\nabla f}{\nabla\left(|f|^{p-2}f\right)}_{{\mathbb C}^{d}}
\geqslant\frac{p}{2}\,\Delta_p(B)|f|^{p-2}|\nabla f|^2.
\end{equation}

\begin{remark}
The expression of the form \eqref{eq: bazilika} appears when one differentiates
the integral of $|\exp(-tL_B)\f|^p$ with respect to $t$ and then integrates by parts. 
Hence it does not come as a surprise that the auxiliary operator $\cI_p$ is a part of the Hessian of $|\zeta|^p$.
More precisely, we have the following formula, valid for $\zeta\in{\mathbb C}\backslash\{0\}, \xi\in{\mathbb C}^d$, which is a reformulation of \cite[(5.5)]{CD-DivForm}:
\begin{equation}
\label{eq: ruzmarin}
d^2F_p(\zeta)\xi=\frac{p^2}{2}|\zeta|^{p-2}
\sign{\zeta}\cdot\cI_p(\sign\bar{\zeta}\cdot\xi).
\end{equation}
Here $F_p(\zeta)=|\zeta|^p$ 
and $d^2F_p(\zeta)$ is the Hessian of $F_p$, calculated at 
$\zeta\in{\mathbb C}\equiv\R^2$ and applied to $\xi\in{\mathbb C}^d\equiv\R^d\times\R^d$.
Now \eqref{eq: kamilica} gives
$
p\nabla\left(|f|^{p-2}f\right)=d^2F_p(f)(\nabla f).
$
\end{remark}

\begin{proof}[Proof of Theorem \ref{t: wabaus}]
We will use Nittka's invariance criterion (Theorem \ref{t: Nittka theorem}).
Under our assumptions on $\phi$, the form $\gotb:=e^{i\phi}\gota$ falls into the framework of Nittka's criterion, by Theorem~\ref{t: majke cigan}. The operator associated with $\gotb$ is $e^{i\phi}L_{A,V}$.  

In order to apply Theorem \ref{t: Nittka theorem}, we must check the following: 
\begin{enumerate}[i)]
\item
$\cD(\gotb)=\cD(\a)$ is invariant under $\bP_{B^p}$;
\item
$\gotb$ is $L^p$-dissipative. 
\end{enumerate}

We us start with i). Let $-\Delta+V$ denote the operator associated with the form $\a_{I,V,\oV}$. 
By the basic assumption \eqref{eq: invariance} and
\cite[Theorem 4.31 {\it 2)}]{O}, the semigroup 
$$
\left(e^{-t(-\Delta+V)}\right)_{t>0}
$$
is contractive on $L^\infty(\Omega)$, and thus, by interpolation with the $L^2$-estimates, on $L^p(\Omega)$ for all $p\in[1,\infty]$. Hence Nittka's Theorem \ref{t: Nittka theorem} gives that $\cD(\a_{I,V,\oV})$ is invariant under $\bP_{B^p}$. Now use that $\cD(\a_{I,V,\oV})=\cD(\a)=\cD(\gotb)$.

The statement ii) follows from the (weak) $p$-ellipticity of $e^{i\phi}A$ virtually without changing the argument from \cite{CD-DivForm}. Indeed, 
if $u\in\cD_{p}(\gotb)$, we get from \eqref {eq: knjiga}, applied with $B=e^{i\phi}A$, that
$$
\aligned
\Re\gotb
\left(u,|u|^{p-2}u\right)
&=
\Re\sk{e^{i\phi}A\nabla u}{\nabla\left(|u|^{p-2}u\right)}_{L^2(\Omega)}+\cos\phi\int_\Omega V|u|^p\\
\endaligned
$$
is a sum of two nonnegative terms. 
\end{proof}

\begin{proof}[Proof of Corollary \ref{c: Marof}]
By the continuity of $\e\mapsto\Delta_r(e^{i\e}A)$, see \cite[Section 5.4]{CD-DivForm}, and monotonicity and symmetry properties of $r\mapsto\Delta_r(e^{i\e}A)$, see \cite[Corollary 5.16 and Proposition 5.8]{CD-DivForm}, there exists $\theta=\theta(p,A)>0$ such that $\Delta_{r}(e^{i\e} A)>0$ for all $\e\in [-\theta,\theta]$ and all $r>1$ satisfying $|1-2/r|\leq|1-2/p|$. The contractivity part now follows from Theorem~\ref{t: wabaus} and the relation 
$$
T^{A,V}_{te^{i\e}}=\exp\left(-te^{i\e}L_{A,V}\right), 
$$
whereupon analyticity is a consequence of a standard argument \cite[Theorem II.4.6]{EN}. 
\end{proof}

\section{Proof of Theorem \ref{t: bilinemb}}

In proving Theorem \ref{t: bilinemb} we will combine and enhance the following tools: 
\begin{enumerate}
\item
contractivity and analyticity properties of the semigroups $T_t$ on $L^p$ \cite{CD-DivForm},
\item
convexity properties of the appropriate {\it Bellman function} \cite{DV-Sch, DV-Kato, CD-mult, CD-DivForm},
\item
analysis of the {\it heat flow} associated with the regularized Bellman function \cite{CD-mult, CD-DivForm, CD-Mixed}.
\end{enumerate}

The first item was already settled in Theorem \ref{t: wabaus}.
We treat the remaining two main steps in separate sections as follows.

\subsection{Bellman function}
Unless specified otherwise, we assume everywhere in this section that $p\geqslant 2$ and $q=p/(p-1)$. Let $\delta>0$. 
The Bellman function we use is the function $ \cQ= \cQ_{p,\delta}:{\mathbb C}\times{\mathbb C}\longrightarrow\R_+$ defined by 
\begin{equation}
\label{eq: rupkina}
 \cQ(\zeta,\eta):=
|\zeta|^p+|\eta|^{q}+\delta
\left\{
\aligned
& |\zeta|^2|\eta|^{2-q} & ; & \ \ |\zeta|^p\leqslant |\eta|^q\\
& \frac{2}{p}\,|\zeta|^{p}+\left(\frac{2}{q}-1\right)|\eta|^{q}
& ; &\ \ |\zeta|^p\geqslant |\eta|^q\,.
\endaligned\right.
\end{equation}
This function is due to Nazarov and Treil. 
See \cite{CD-DivForm} or \cite{CD-Mixed} for an up-to-date account on previous appearances of $\cQ$ in the literature.

It is a direct consequence of the above definition that the function $\cQ$ belongs to $C^1({\mathbb C}^2)$ and is of order $C^2$ everywhere {\it except} on the set
$$
\Upsilon=\mn{(\zeta,\eta)\in  {\mathbb C}\times{\mathbb C}}{(\eta=0)\vee (|\zeta|^p=|\eta|^q)}\,.
$$

We shall use the notation from \cite[Section 2.2]{CD-DivForm} to denote the {\it generalized Hessians} $H_{\cQ}^{(A,B)}[v;\omega]$ and $H_{F}^{A}[\zeta;\xi]$. 
Vaguely speaking, if 
$N\in\N$, 
$\omega$ is a $N$-tuple of complex numbers,
$X$ a $N$-tuple of vectors from ${\mathbb C}^d$,
${\mathbf A}$ a $N$-tuple of matrices from ${\mathbb C}^{d\times d}$
and $\Phi:{\mathbb C}^N\rightarrow \R$ is of class $C^2$,
then one has
$$
H_\Phi^{\mathbf A}[\omega;X]=\sk{d^2\Phi(\omega)X}{{\mathbf A}X}_{ ({\mathbb C}^{d})^N }.
$$
In comparison with the exact definition \cite{CD-DivForm} 
we omitted here the tensorization of the Hessian $d^2\Phi(\omega)$ with the $d\times d$ identity matrix, as well as appropriate identifications of complex objects (vectors, matrices) with their real counterparts.

Here is a stronger version of \cite[Theorem 5.2]{CD-DivForm} suited for potentials. It evokes \cite[Theorem 3]{DV-Sch}, where similar properties of $\cQ$ were proved, also for the purpose of treating Schr\"odinger operators. 
The main distinction is that here we treat arbitrary complex matrix functions $A,B$, while in \cite[Theorem 3]{DV-Sch} we only addressed the case $A=B\equiv I$. 
The reader is also referred to \cite{DTV}.

\begin{thm}
\label{t: trica}
Choose $p\geq2$ and $A,B\in\cA_p(\Omega)$.
Then there exists $\delta\in(0,1)$ such that $\cQ= \cQ_{p,\delta}$ as in \eqref{eq: rupkina} admits the following property: 

For any $v=(\zeta,\eta)\in{\mathbb C}^2\setminus\Upsilon$ 
there exists $\tau=\tau(v)>0$ such that, a.e. $x\in\Omega$, 
\begin{itemize}
\item
$H_{\cQ}^{(A(x),B(x))}[v;\omega] 
\geqsim
\tau
|\omega_1|^2
+\tau^{-1}
|\omega_2|^2
$
for any
 $\omega=(\omega_1,\omega_2)\in{\mathbb C}^{d}\times{\mathbb C}^{d}$, 
 and
\item
$(\pd_\zeta \cQ)(\zeta,\eta)\cdot\zeta
\,\geqsim\, \tau |\zeta|^2$
and 
$(\pd_\eta \cQ)(\zeta,\eta)\cdot\eta
\,\geqsim\, \tau^{-1} |\eta|^2$, 
\end{itemize}

with the implied constants depending $p,A,B$, but not on the dimension $d$.
\end{thm}

\begin{remark} 
It seems that, for arbitrary functions, the first property in general does {\it not} imply the second one {\it with the same $\tau$}, even in the case $A=B\equiv I$. See \cite[Section 2.2]{DV-Sch}.
\end{remark}

In order to prove Theorem \ref{t: trica} we start with a pair of elementary equivalences
which are variants of \cite[Lemma 5.24]{CD-DivForm}.

\begin{lemma}
\label{l: Jamnicka}
Suppose that $\alpha,\beta,\gamma\in\R$. The following statements are equivalent:
\begin{enumerate}
\item
\label{eq: bolecino}
There exists $\tau>0$ such that 
$\alpha x^2-2\beta xy+\gamma y^2\geq \tau x^2+\tau^{-1}y^2$ 
for all $x,y\in\R$.
\item
\label{eq: pijemo}
$\alpha,\gamma>0$ and $\sqrt{\alpha\gamma}-|\beta|\geq1$. 
\end{enumerate}
\end{lemma}

\begin{proof}
{\it(\ref{eq: bolecino})} $\Rightarrow$ {\it(\ref{eq: pijemo})}: 
By taking $x\ne0=y$ and $x=0\ne y$ we get
\begin{equation}
\label{eq: amygdala}
\alpha\geq\tau\geq\frac1\gamma>0.
\end{equation}
The assumption {\it(\ref{eq: bolecino})} implies that
$
\alpha x^2-2\beta xy+\gamma y^2\geq 2|xy|
$
and hence 
$$
\alpha x^2+\gamma y^2\geq 2(|\beta|+1)|xy|
\hskip40pt
\forall x,y\in\R\,.
$$
Diviging by $y^2$ and writing $t=|x/y|$ we get 
$$
\alpha t^2-2(|\beta|+1)t+\gamma\geq0\,.
\hskip40pt
\forall t>0\,.
$$
Clearly the above inequality is then also valid for $t\leq0$, and thus for all $t\in\R$. Of course, this is possible if and only if $(|\beta|+1)^2 -\alpha\gamma\leq0$.

\medskip
{\it(\ref{eq: pijemo})} $\Rightarrow$ {\it(\ref{eq: bolecino})}:
Define $\tau:=\sqrt{\alpha/\gamma}$. Our assumption implies 
\eqref{eq: amygdala}.
Then for any $x,y\in\R$ we have
$$
\aligned
\alpha x^2-2\beta& xy+\gamma y^2 -\tau x^2-\tau^{-1}y^2\\
& = (\alpha-\tau)x^2+\left(\gamma-\tau^{-1}\right)y^2-2\beta xy\\
& \geq 2\sqrt{\left(\alpha-\tau\right)\left(\gamma-\tau^{-1}\right)}|xy|-2|\beta||xy|\\
&=\frac{2|xy|}{\sqrt{\left(\alpha-\tau\right)\left(\gamma-\tau^{-1}\right)}+|\beta|}
\left(\big(\alpha-\tau\big)\big(\gamma-\tau^{-1}\big)-\beta^2\right).
\endaligned
$$ 
Since
$$
\big(\alpha-\tau\big)\big(\gamma-\tau^{-1}\big)-\beta^2
=
\left(\sqrt{\alpha\gamma}-1\right)^2-\beta^2\geq0,
$$
our proof is complete.
\end{proof}

\begin{cor}
\label{c: shinjuku}
Suppose $a,b,c\in\R$. The 
following conditions are equivalent: 
\begin{itemize}
\item
there exist $C,\tau>0$ such that 
$ax^2-2bxy+cy^2\geq C(\tau x^2+\tau^{-1}y^2)$ for all $x,y\in\R$;
\item
$a,c>0$ and $ac-b^2>0$. 
\end{itemize}
In this case the largest admissible choice for $C$ is
$C=\sqrt{ac}-|b|$. Moreover, we may take $\tau=\sqrt{a/c}$.
\end{cor}

\begin{proof}
Divide the inequality from the first statement by $C$ and use Lemma \ref{l: Jamnicka}.
\end{proof}

\begin{proof}[Proof of Theorem \ref{t: trica}]
We follow, and adequately modify, the proof of \cite[Theorem 5.2]{CD-DivForm}, which was in turn modelled after the proof of \cite[Theorem 3]{DV-Sch}.

Denote $F_p(\zeta)=|\zeta|^p$ for $\zeta\in{\mathbb C}$. When $p=2$, the Bellman function reads $\cQ(\zeta,\eta)=(1+\delta)F_2(\zeta)+F_2(\eta)$ for all $\zeta,\eta\in{\mathbb C}$. Therefore, by \cite[Lemma 5.6]{CD-DivForm},
\begin{equation}
\label{eq: hvanchkara}
\aligned
H_{\cQ}^{(A,B)}[v;\omega]
&=(1+\delta)H_{F_2}^A[\zeta;\omega_1]+H_{F_2}^B[\eta;\omega_2]\\
&=2(1+\delta)\Re\sk{A\omega_1}{\omega_1}+2\Re\sk{B\omega_2}{\omega_2}\\
&\geq 2\lambda(A,B)\left(|\omega_1|^2+|\omega_2|^2\right).
\endaligned
\end{equation}
On the other hand, trivial calculations show that 
\begin{equation}
\label{eq: churchela}
\text{ if }\
\Phi=F_r\otimes F_s\ 
\text{ with }\
r,s\geq0,\ 
\text{ then }\
\left\{
\renewcommand{\arraystretch}{1.2}
\begin{array}{l}
(\pd_\zeta\Phi)(\zeta,\eta)\cdot\zeta=(r/2)\Phi(\zeta,\eta)\\
(\pd_\eta\Phi)(\zeta,\eta)\cdot\eta=(s/2)\Phi(\zeta,\eta).
\end{array}
\renewcommand{\arraystretch}{1}
\right.
\end{equation}
Therefore 
\begin{equation}
\label{eq: kindzmarauli}
\aligned
&(\pd_\zeta \cQ)(\zeta,\eta)\cdot\zeta=(1+\delta)|\zeta|^2\\
&(\pd_\eta \cQ)(\zeta,\eta)\cdot\eta=|\eta|^2.
\endaligned
\end{equation}
By combining \eqref{eq: hvanchkara} and \eqref{eq: kindzmarauli} we prove the theorem in the case $p=2$ with $\tau=1$.

\medskip
Now consider the case $p>2$.
Denote $\textsf{u}=|\zeta|$, $\textsf{v}=|\eta|$, $\textsf{ A}=|\omega_1|$, $\textsf{ B}=|\omega_2|$. 
Recall the notation \eqref{eq: kabuto}. We divide ${\mathbb C}^2\setminus\Upsilon$ into two natural subdomains in which we analyze the gradients and Hessians of $\cQ$ separately.

\medskip
First assume that \framebox{$\textsf{u}^p> \textsf{v}^q>0$}. 
Then, by \cite[proof of Theorem 5.2]{CD-DivForm},
$$
H_{\cQ}^{(A,B)}[v;\omega] 
\geqslant 
\frac{pq\Delta_p(A,B)}2
\left[
(p-1) \textsf{u}^{p-2}\textsf{A}^2
+(q-1) \textsf{v}^{q-2}\textsf{B}^2
\right]\,.
$$
So in this case we may take $\tau=(p-1) \textsf{u}^{p-2}$, as in \cite[proof of Theorem 3]{DV-Sch}.

Regarding the last pair of estimates, since here $\cQ_{p,\delta}$ is a linear combination of functions $F_p\otimes F_0$ and $F_0\otimes F_q$, it follows from \eqref{eq: churchela} that, similarly to \eqref{eq: kindzmarauli},
\begin{equation*}
\aligned
&(\pd_\zeta \cQ)(\zeta,\eta)\cdot\zeta=C_1(p,\delta)|\zeta|^p\\
&(\pd_\eta \cQ)(\zeta,\eta)\cdot\eta=C_2(p,\delta)|\eta|^q.
\endaligned
\end{equation*}
Therefore, since $\textsf{v}^{q-2}> \textsf{u}^{2-p}>0$, with $\tau=(p-1) \textsf{u}^{p-2}$ we get 
\begin{equation}
\aligned
&(\pd_\zeta \cQ)(\zeta,\eta)\cdot\zeta=\widetilde{C_1}(p,\delta)\tau|\zeta|^2\\
&(\pd_\eta \cQ)(\zeta,\eta)\cdot\eta\geq\widetilde{C_2}(p,\delta)\tau^{-1}|\eta|^2.
\endaligned
\end{equation}

\medskip
Suppose now that \framebox{$\textsf{u}^p<\textsf{v}^q$}. 
Extend 
the definition \eqref{eq: kabuto} to all $p\in(0,\infty]$. Then, as in \cite[proof of Theorem 5.2]{CD-DivForm},
\begin{equation*}
\aligned
H_{\cQ}^{(A,B)}[v;\omega] 
\geq
2\delta
\left(
\lambda(A) \textsf{v}^{2-q} \textsf{A}^2
-2(2-q)\Lambda(A,B) \textsf{A}\textsf{B}
+\frac\Gamma4
\textsf{v}^{q-2}\textsf{B}^2
\right),
\endaligned
\end{equation*}
where
\begin{equation*}
\label{eq: oher}
\Gamma=\frac{q^2\Delta_q(B)}\delta+(2-q)^2\Delta_{2-q}(B)\,.
\end{equation*}
Since $\Delta_p(B)>0$, we have that $\Gamma$ grows to infinity as $\delta
\downarrow0$. Since we also have $\lambda(A) 
>0$, there exists $\delta=\delta(p,A,B)>0$ such that 
$$
\frac{\lambda(A)\Gamma}4>\left[(2-q)\Lambda(A,B)\right]^2,
$$ 
which through Corollary \ref{c: shinjuku} implies the existence of $\tau>0$ that accommodates the first requirement of Theorem \ref{t: trica}. Moreover, we may take $\tau=D\textsf{v}^{2-q}$, where $D=2\sqrt{\lambda(A)/\Gamma}$.

Now consider the gradient estimates. In the domain $\{\textsf{u}^{p}<\textsf{v}^q\}$ we have 
$$
\cQ=F_{p}\otimes F_{0}+F_{0}\otimes F_{q}+\delta F_2\otimes F_{2-q}.
$$ 
Again \eqref{eq: churchela} implies that
with (the above chosen) $\tau=D\textsf{v}^{2-q}$ we get
\begin{equation*}
\aligned
(\pd_\zeta \cQ)(\zeta,\eta)\cdot\zeta
  &=\frac p2|\zeta|^p+\delta|\zeta|^2|\eta|^{2-q}
    &&\hskip -11pt \geq\delta|\eta|^{2-q}|\zeta|^2
       && \hskip -11pt 
\geq\frac\delta D\tau|\zeta|^2\\
(\pd_\eta \cQ)(\zeta,\eta)\cdot\eta
  &=\frac q2|\eta|^q+\delta\cdot\frac{2-q}2|\zeta|^2|\eta|^{2-q}
    &&\hskip -11pt \geq\frac q2|\eta|^{q-2}|\eta|^2
      && \hskip -11pt 
\geq\frac{Dq}2\tau^{-1}|\eta|^2.
\endaligned
\end{equation*}
This finishes the proof of the theorem.
\end{proof}

\subsubsection*{{\bf Identification operators}}
\label{s: grand sonata luganskij}
We will explicitly identify ${\mathbb C}^{d}$ with $\R^{2d}$. 
For each $d\in\N$ consider the operator
$\cV_{d}:{\mathbb C}^{d}\rightarrow\R^{d}\times\R^{d}$, defined by 
$$
\cV_{d}(\alpha+i\beta)=
(\alpha,\beta).
$$
 One has, for all  $z,w\in{\mathbb C}^{d}$,
 \begin{equation}
\label{eq: angelis - romance - hatmulin}
\Re\!\sk{z}{w}_{{\mathbb C}^{d}}=\sk{\cV_n(z)}{\cV_n(w)}_{\R^{2d}}
\,.
\end{equation}
If $(\omega_{1},\omega_{2})\in{\mathbb C}^{d}\times{\mathbb C}^{d}$ then
$
\cV_{2n}(\omega_{1},\omega_{2})
=(\Re\omega_1, \Re\omega_2, \Im\omega_1, \Im\omega_2)
\in(\R^{d})^4.
$
On ${\mathbb C}^{d}\times{\mathbb C}^{d}$ define another identification operator, $\cW_{2d}:{\mathbb C}^{d}\times{\mathbb C}^{d}\rightarrow (\R^{d})^4$, by
$$
\cW_{2n}(\omega_{1},\omega_{2})
=(\cV_n(\omega_1),\cV_n(\omega_2))
=(\Re\omega
_1, \Im\omega_1, \Re\omega_2, \Im\omega_2). 
$$
When the dimensions of the spaces on which the identification operators act are clear, we will sometimes omit the indices and instead of $\cV_n,\cW_m$ only write $\cV,\cW$.
For example,
\begin{equation}
\cW(\zeta,\eta)=\left(\cV(\zeta),\cV(\eta)\right)
\hskip 30pt
\forall \zeta,\eta\in{\mathbb C}.
\end{equation}

For functions $\Phi$ on spaces ${\mathbb C}^k$ we will sometimes use their ``pullbacks'' defined on  
$\R^{2k}$, namely 
$$
\Phi_\cV=\Phi\circ\cV^{-1}
\hskip 30pt
\text{or}
\hskip 30pt
\Phi_\cW=\Phi\circ\cW^{-1}
\,.
$$

\subsubsection*{{\bf Regularization of $\cQ$}}

Denote by $*$ the convolution in $\R^4$ and let $(\f_\kappa)_{\kappa>0}$ be a nonnegative, smooth and compactly supported approximation of the identity on $\R^4$. Explicitly, $\f_\kappa(y)=\kappa^{-4}\f(y/\kappa)$, where $\f$ is smooth, nonnegative, radial, of integral 1, and supported in the closed unit ball 
in $\R^4$. If $\Phi:{\mathbb C}^2\rightarrow \R$, define $\Phi*\f_\kappa=(\Phi_\cW*\f_\kappa)\circ \cW:{\mathbb C}^2\rightarrow \R$. Explicitly, for $ v\in{\mathbb C}^2$,
\begin{equation}
\label{eq: big deal}
\aligned
(\Phi*\f_\kappa)( v)
&=\int_{\R^4}\Phi_\cW\left(\cW( v)-y\right)\f_\kappa(y)\,dy\\
&=\int_{\R^4}\Phi\left(v-\cW^{-1}(y)\right)\f_\kappa(y)\,dy.
\endaligned
\end{equation}

Now we can formulate a version of Theorem \ref{t: trica} for the mollifications $\cQ*\f_\kappa$, in the fashion of \cite[Theorem 4]{DV-Sch}. It also strengthens \cite[Corollary 5.5]{CD-DivForm}.

\begin{thm}
\label{t: regular Q}
Choose $p\geq2$ and $A,B\in\cA_p(\Omega)$.
Let $\delta\in(0,1)$ and function $\tau:{\mathbb C}^2\setminus\Upsilon\rightarrow(0,\infty)$ be as in Theorem \ref{t: trica}. Then for $ \cQ= \cQ_{p,\delta}$ as in \eqref{eq: rupkina} and any $ v=(\zeta,\eta)\in{\mathbb C}^2$ we have, 
for a.e. $x\in\Omega$ and every 
 $\omega=(\omega_1,\omega_2)\in{\mathbb C}^{d}\times{\mathbb C}^{d}$, 
\begin{equation}
\label{eq: big fish theory}
H_{ \cQ*\f_\kappa}^{(A(x),B(x))}[ v;\omega]
\,\geqsim\,
\left(\tau*\f_\kappa\right)( v)\cdot|\omega_1|^2
+\left(\tau^{-1}*\f_\kappa\right)( v)\cdot|\omega_2|^2,
\end{equation}
with the implied constant depending on 
$p,A,B$,
but not on the dimension $d$.
\end{thm}

\begin{proof}
As in \cite[proof of Corollary 5.5]{CD-DivForm} we obtain, for $ v\in{\mathbb C}^2$, $\omega\in{\mathbb C}^{d}\times{\mathbb C}^{d}$ and $\kappa>0$,
\begin{equation*}
H_{ \cQ*\f_\kappa}^{(A,B)}[ v;\omega]
=\int_{\R^4}H_{ \cQ}^{(A,B)}\left[ v-\cW^{-1}( y);\omega\right]\f_\kappa( y)\,d y\,.
\end{equation*}
The first estimate of Theorem \ref{t: trica} now gives 
\begin{equation*}
H_{ \cQ*\f_\kappa}^{(A,B)}[ v;\omega]
\geqsim
\int_{\R^4}
\left(\tau( v-\cW^{-1}( y))|\omega_1|^2+\tau^{-1}( v-\cW^{-1}( y))|\omega_2|^2\right)
\f_\kappa( y)\,dy.
\end{equation*}
By recalling the convention \eqref{eq: big deal}, we see that we just obtained \eqref{eq: big fish theory}.
\end{proof}

\subsection{Heat flow}
\label{s: Heat flow}
As announced before, we prove the bilinear embedding by means of a {\it heat flow} technique applied to the Nazarov--Treil function $\cQ$. We follow the outline of the method in \cite{CD-DivForm, CD-Mixed}, where we proved the theorem for $V=W=0$. The presence of nonzero potentials, considered in this paper, calls for settling a couple of technical problems which do not appear in the homogeneous case.
As a historical note we mention that the early versions of the heat-flow method associated with Bellman functions go back to the papers by Petermichl--Volberg \cite{PV} and Nazarov--Volberg \cite{NV}.

Proving bilinear embedding on arbitrary domains $\Omega$ \cite{CD-Mixed}, as opposed to proving it for $\Omega=\R^{d}$ \cite{CD-DivForm}, requires a major modification of the heat-flow argument. See \cite[Section 1.4]{CD-Mixed} for explanation. The gist of the problem is to justify {\it integration by parts}, which was overcome in \cite{CD-Mixed} by approximating $\cQ$ by a specifically constructed sequence of functions, see \cite[Theorem 16]{CD-Mixed}.

For $f,g\in (L^p\cap L^q)(\Omega)$ and $A,B\in\cA(\Omega)$ define
$$
\cE(t)=\int_{\Omega} \cQ\left(T^{A,V,\oV}_{t} f, T^{B,W,\oW}_{t}g\right),\quad t>0.
$$
Known estimates of $\cQ$ and its gradient \cite[Theorem 4]{CD-Riesz}
and the analyticity of $(T^{A}_{t})_{t>0}$ and $(T^{B}_{t})_{t>0}$ (see Theorem~\ref{t: wabaus}) imply that $\cE$ is continuous on $[0,\infty)$ and differentiable on $(0,\infty)$ with a continuous derivative. As in our previous works involving the heat flow, our aim is to prove two-sided estimates of
\begin{equation}
\label{eq: Messiah}
-\int_0^\infty\cE'(t)\,dt
\end{equation}
which will then, in a by now familiar Bellman-heat fashion, see e.g. \cite{CD-DivForm,CD-Mixed} and the references there, merge into bilinear embedding. Regarding the {\bf upper estimates} of \eqref{eq: Messiah}, we use upper pointwise estimates on $\cQ$ (see, for example \cite[Proposition 5.1]{CD-DivForm}) to get
\begin{equation}
\label{eq: sol}
-\int_0^\infty\cE'(t)\,dt\leq\cE(0)\,\leqsim\,\nor{f}_p+\nor{g}_q.
\end{equation}

Now we turn to {\bf lower estimates}. For
\begin{equation}
\label{eq: vw}
(v,w)=\left(T^{A,V,\oV}_{t} f, T^{B,W,\oW}_{t} g\right)
\end{equation}
we have, by Corollary \ref{c: Marof}, that $v,w\in L^p\cap L^q$ and 
\begin{equation}
\label{eq: ostanes}
\aligned
-\cE'(t)
=2\,\Re\int_{\Omega}
\left(
\sk{(\pd_{\bar\zeta} \cQ)(v,w)}{L_{A,V}v}_{\mathbb C}
+
\sk{(\pd_{\bar\eta} \cQ)(v,w)}{L_{B,W}w}_{\mathbb C}
\right).
\endaligned
\end{equation}
See \cite[Section 6.1]{CD-Mixed} for more details on how to justify \eqref{eq: ostanes}.

\subsection{Special case: bounded potentials}
\label{s: BP}
First we prove the bilinear embedding under additional assumption that $V,W$ are (nonnegative and) {\it essentially bounded}. In that case, $\cD(L_{A,V})=\cD(L_{A,0})$ and for $u\in\cD(L_{A,V})$ we have
$$
L_{A,V}u=L_{A,0}u+Vu
$$
and the same for $B,W$. Consequently, \eqref{eq: ostanes} gives 
\begin{equation}
\label{eq: sum}
-\cE'(t)=I_1+I_2,
\end{equation}
where 
$$
\aligned
I_1&=2\,\Re\int_{\Omega}
\left(
\sk{(\pd_{\bar\zeta} \cQ)(v,w)}{L_{A,0}v}_{{\mathbb C}}
+
\sk{(\pd_{\bar\eta} \cQ)(v,w)}{L_{B,0}w}_{{\mathbb C}}
\right)\\
I_2&=2\int_\Omega\left[V\left(\pd_\zeta \cQ\right)(v,w)\cdot v+W\left(\pd_\eta \cQ\right)(v,w)\cdot w\right].
\endaligned
$$
The integral $I_2$ is absolutely convergent by the upper pointwise estimates of the gradient of $\cQ$ (see, for example \cite[Proposition 5.1]{CD-DivForm}), H\"older's inequality, the contractivity of the semigroups from \eqref{eq: vw}  on $L^p$ and $L^q$, respectively (cf. Theorem \ref{t: wabaus}), and the assumption that $V,W\in L^\infty(\Omega)$.

We will estimate the terms $I_1,I_2$ separately. 

\subsubsection{Estimate of $I_1$} As in \cite[Section 6.1]{CD-Mixed} we get
$$
I_1\,\geq\,
\liminf_{\kappa\searrow 0}
\int_{\Omega}H^{(A,B)}_{\cQ*\f_\kappa}\left[\left(v,w\right);\left(\nabla v,\nabla w\right)\right].
$$
(It is here that we needed to restrict the choice of $\oV,\oW$ to (a)-(c) from Section \ref{s: Karlstejn}.) Next we apply Theorem \ref{t: regular Q} for 
$$
I_1\,\geqsim\,
\liminf_{\kappa\searrow 0}
\int_{\Omega}
\left[\left(\tau*\f_\kappa\right)(v,w)\cdot|\nabla v|^2
+\left(\tau^{-1}*\f_\kappa\right)(v,w)\cdot|\nabla w|^2\right].
$$
We see from the proof of Theorem \ref{t: trica} that one can choose the function $\tau$ to be
continuous on ${\mathbb C}^2\backslash\Upsilon$. Now Fatou's lemma gives
$$
I_1\,\geqsim\,
\int_{\Omega}
\left[\tau(v,w)\cdot|\nabla v|^2
+\tau^{-1}(v,w)\cdot|\nabla w|^2\right].
$$
\subsubsection{Estimate of $I_2$} 
Using that $V,W$ are nonnegative, we get from Theorem \ref{t: trica}
$$
I_2\,\geqsim\,\int_\Omega\left(\tau(v,w)\cdot V|v|^2+\tau(v,w)^{-1}\cdot W|w|^2\right).
$$
The last two estimates give, together with \eqref{eq: sum} 
\begin{equation}
\label{eq: papar}
-\cE'(t)\
\geqsim\,\int_\Omega
\sqrt{|\nabla v|^2+V|v|^2}\sqrt{|\nabla w|^2+W|w|^2}.
\end{equation}

\subsubsection{Summary}
By merging \eqref{eq: sol} and \eqref{eq: papar} we get 
$$
\aligned
\int_0^\infty\int_\Omega
\sqrt{|\nabla T_{t}^{A,V,\oV} f|^2+V|T_{t}^{A,V,\oV} f|^2}&\sqrt{|\nabla T_{t}^{B,W,\oW} g|^2+W|T_{t}^{B,W,\oW} g|^2}\,dx\,dt\\
&\hskip 50pt \,\leqsim\,
\nor{f}_p+\nor{g}_q\,.
\endaligned
$$
Now use the standard trick and replace $f$ by $\mu f$ and $g$ by $\mu^{-1}g$, with $\mu>0$, and minimize in $\mu$. This finishes the proof of Theorem \ref{t: bilinemb} in the case of 
{\it bounded} $V,W$.

\subsection{General case: unbounded potentials} 
Theorem~\ref{t: bilinemb} will follow from the special case of bounded potentials already proved in Section~\ref{s: BP}, once we prove the following approximation result.

Let $U\in L^{1}_{\rm loc}(\Omega)$ be a nonnegative function. For each $n\in\N$ define
$$
U_{n}:=\min\{U,n\}.
$$
We also 
set $U_{\infty}=U$.
\begin{thm}
\label{t: 1a}
For all $f\in L^{2}(\Omega)$, $A\in \cA(\Omega)$, $U\in L_{loc}^1(\Omega)$ and all $t>0$ we have
$$
\aligned
\nabla T^{A,U_{n}}_{t}f 
& \rightarrow \nabla T^{A,U}_{t}f 
&&\text{in }L^{2}(\Omega;{\mathbb C}^{d}),\\
U^{1/2}_{n} T^{A,U_{n}}_{t}f
& \rightarrow U^{1/2}T^{A,U}_{t}f 
&& \text{in }L^{2}(\Omega)
\endaligned
$$
as $n\rightarrow\infty$.
\end{thm}
The proof will be given in Section~\ref{s: Teske boje}. First we need a few technical results.

\begin{notation}
Until the end of this chapter we will work with a single matrix function $A$. Therefore, in order to make the text more readable, we will from now on omit $A$ in the notation for the operators and semigroups. For example, we will write 
$T_t^U$ instead of $T_t^{A,U}$ and $L_U$ instead of $L_{A,U}$. 
\end{notation}

Recall that $\nu(A)$ was defined on page \pageref{Borodin Quartet}. 
It then follows from the positivity of $U$ and the estimate \cite[(1.26)]{O} that the operators $L_{U_{n}}$, $n\in\N\cup\{\infty\}$, are uniformly sectorial of angle $\nu=\nu(A)$ in the sense that
\begin{equation}
\label{eq: 3}
\norm{(\zeta-L_{U_{n}})^{-1}}{2}\leq \frac{1}{{\rm dist}(\zeta,\overline{\bS}_{\nu })},\quad \forall \zeta\in {\mathbb C}\setminus \overline{\bS}_{\nu }.
\end{equation}

We will use the next lemma whose proof is based on an idea of Ouhabaz \cite{Ouhabaz1995} that we learnt from \cite{OuhabazBailey}.
\begin{lemma}
\label{l: 2a}
For all $f\in L^{2}(\Omega)$ and all $s>0$ we have
\begin{equation}
\label{eq: 5}
\left(s+L_{U_{n}}\right)^{-1}f\longrightarrow \left(s+L_{U}\right)^{-1}f \quad \text{in }L^{2}(\Omega), \text{ as }n\rightarrow\infty.
\end{equation}
\end{lemma}
\begin{proof}[Sketch of the proof.]
The proof is based on the argument presented in \cite[p. 19-20]{OuhabazBailey}. Let us outline the main steps. 

Recall the definitions  \eqref{eq: mongolska} and \eqref{eq: govedina} and consider 
the sesquilinear forms $\gota=\gota_U$ and $\gota_n:=\gota_{U_n}$. 
We define operations on forms as in \cite[Chapter VI, {\tretamajka \char120}\,1.1]{Kat}.
For $z\in{\mathbb C}$ and $n\in\N$ denote 
$$
\aligned
\gota_z &:=\Re\gota\hskip 5pt+z\,\Im \gota\\
\gota_{n,z} &:=\Re\gota_{n}+z\,\Im \gota_n.
\endaligned
$$
Note that
$\gota=\gota_i$. 
Set $\delta=\delta(A)=\cot\nu(A)$ and 
${\mathcal O}:=\mn{z\in{\mathbb C}}{|\Re z|<\delta}$. 
It can be shown
that if $z\in\cO$ then
$\gota_z$ and $\gota_{n,z}$ are closed sectorial forms.

Let $L_z$ and $L_{n,z}$ be the operators associated with 
$\gota_z$ and $\gota_{n,z}$, respectively. 
For $z\in\cO$ and $s>0$, the operator $s+L_z$ is invertible and 
$\nor{(s+L_z)^{-1}}_2\leq\nor{f}_2/s$,
cf. \eqref{eq: 3}. 
A theorem by Kato \cite[p. 395]{Kat}, see also \cite{VoVoOL}, shows that $z\mapsto (s+L_z)^{-1}$ is holomorphic as a map 
from $\cO$ to the space of bounded linear operators on $L^2(\Omega)$.
The same holds for 
the map $z\mapsto (s+L_{n,z})^{-1}$.

A monotone convergence theorem for sequences of symmetric sesquilinear forms, see \cite[Theorem~3.13a, p. 461]{Kat} and \cite[Theorem~3.1]{Simon1978}, gives that for every $s>0$, $z\in(-\delta,\delta)$ and $f\in L^2(\Omega)$ we have
\begin{equation}
\label{eq: Triangle Offense}
(s+L_{n,z})^{-1}f\rightarrow(s+L_{z})^{-1}f
\hskip 30pt
\text{ in }L^2(\Omega),\text{ as }n\rightarrow\infty.
\end{equation}
A vector-valued version of Vitali's theorem \cite[Theorem~A.5]{ABHbook} implies that
$(s+L_{n,z})^{-1}f\rightarrow (s+L_{z})^{-1}f$ for all $z\in\cO$, and \eqref{eq: 5} follows by taking $z=i$.
\end{proof}

\begin{prop}
\label{p: 1a}
For all $f\in L^{2}(\Omega)$ and all $\zeta\in{\mathbb C}\setminus\overline{\bS}_{\nu }$ we have 
\begin{align}
\left(\zeta-L_{U_{n}}\right)^{-1}f 
&\rightarrow 
\left(\zeta-L_{U}\right)^{-1}f 
&& \text{in } L^{2}(\Omega),
\label{eq: a11a}\\
\nabla\left(\zeta-L_{U_{n}}\right)^{-1}f 
&\rightarrow 
\nabla\left(\zeta-L_{U}\right)^{-1}f 
&& \text{in } L^{2}(\Omega;{\mathbb C}^{d}),
\label{eq: a1a}\\
U^{1/2}_{n}\left(\zeta-L_{U_{n}}\right)^{-1}f
&\rightarrow U^{1/2}\left(\zeta-L_{U}\right)^{-1}f
&& \text{in } L^{2}(\Omega),
\label{eq: a2a}
\end{align}
as $n\rightarrow\infty$.
\end{prop}

\begin{proof}
Recall the notation $U_{\infty}=U$. 
Fix $f\in L^{2}(\Omega)$. For $n\in\N\cup\{\infty\}$ and $\zeta\in {\mathbb C}\setminus \overline{\bS}_{\nu }$ set
$$
u_{n}(\zeta):=\left(L_{U_{n}}-\zeta\right)^{-1}f\in \cD(L_{U_{n}})\subseteq\oV\subseteq H^{1}(\Omega).
$$
By ellipticity of $A$, for every 
$n\in\N\cup\{\infty\}$ and $\zeta\in {\mathbb C}\setminus\overline\bS_{\nu }$ we have
$$
\aligned
\lambda\|\nabla u_{n}\|^{2}_{2}+\|U^{1/2}_{n}u_{n}\|^{2}_{2}
&\leq\Re\int_{\Omega}\left[\sk{A\nabla u_{n}}{\nabla u_{n}}+U_{n}u_{n}\overline{u}_{n}\right]\\
&=\Re\int_{\Omega}\left(L_{U_{n}}u_{n}\right)\overline{u_{n}}\\
&=\Re\int_{\Omega}f\overline{u_{n}} +(\Re\zeta)\int_{\Omega}|u_n|^2\\
&\leq \|f\|_{2}\|u_{n}\|_{2}+|\Re\zeta|\cdot\|u_{n}\|_{2}^2.
\endaligned
$$
Therefore, the uniform sectoriality estimate \eqref{eq: 3} gives,
for all $n\in\N\cup\{\infty\}$ and $\zeta\in {\mathbb C}\setminus\overline\bS_{\nu }$,
\begin{equation}
\label{eq: 3bis}
\norm{u_{n}(\zeta)}{2}+\|\nabla u_{n}(\zeta)\|_{2}+\|U^{1/2}_{n}u_{n}(\zeta)\|_{2}
\leq C_{\lambda,\nu}(\zeta)
 \norm{f}{2},
\end{equation}
where $C_{\lambda,\nu}(\zeta)>0$ is continuous in $\zeta$.

\medskip
Now temporarily fix $s>0$ and set 
$$
\aligned
u_n &= \hskip 3pt u_n(-s)\\
u &= u_\infty(-s).
\endaligned
$$
By \eqref{eq: 3bis}, the sequence $\big(u_{n},U^{1/2}_{n}u_{n}\big)_{n\in\N}$ is bounded in $H^{1}(\Omega)\times L^{2}(\Omega)$, hence it admits a weakly convergent subsequence. That is, there exist a subsequence of indices $(n_{k})_{k\in\N}$ and functions $w\in H^{1}(\Omega)$ and $\widetilde{w}\in L^{2}(\Omega)$ such that
$$
\aligned
u_{n_{k}} &\rightharpoonup  w\quad \text{ in } H^{1}(\Omega),\\
U^{1/2}_{n_{k}}u_{n_{k}} & \rightharpoonup  \widetilde{w}\quad \text{ in } L^{2}(\Omega).
\endaligned
$$
Here the symbol $\rightharpoonup$ denotes weak convergence (that is, convergence in the weak topology).
Lemma~\ref{l: 2a} reads 
\begin{equation}
\label{eq: 6}
\lim_{n\rightarrow\infty}u_{n}= u_{\infty}\quad \text{in }L^{2}(\Omega),\quad \forall s>0.
\end{equation}
which implies that $w=u$.

Fix $\varphi\in C^{\infty}_{c}(\Omega)$. 
Since $U\in L^{1}_{\rm loc}(\Omega)$, we have $U^{1/2}\varphi\in L^{2}(\Omega)$, thus, by \eqref{eq: 6},
$$
\aligned
\int_{\Omega}\widetilde{w}\varphi
& =\lim_{k\rightarrow\infty}\int_{\Omega}U^{1/2}_{n_{k}}u_{n_{k}}\varphi\\
& =\lim_{k\rightarrow\infty}\int_{\Omega}u_{n_{k}}U^{1/2}\varphi
  +\lim_{k\rightarrow\infty}\int_{\Omega}u_{n_{k}}\left(U^{1/2}_{n_{k}}-U^{1/2}\right)\varphi\\
& =\int_{\Omega}uU^{1/2}\varphi
  +\lim_{k\rightarrow\infty}\int_{\Omega}u_{n_{k}}\left(U^{1/2}_{n_{k}}-U^{1/2}\right)\varphi.
\endaligned
$$
Now,
$$
\mod{U^{1/2}_{n_{k}}-U^{1/2}}^{2}|\varphi|^{2}\leq U|\varphi|^{2}\in L^{1}(\Omega) 
$$
and $U_{n_{k}}\rightarrow U$ pointwise. Hence, by the dominated convergence theorem and \eqref{eq: 3},
$$
\lim_{k\rightarrow\infty}\int_{\Omega}\left|u_{n_{k}}\left(U^{1/2}_{n_{k}}-U^{1/2}\right)\varphi\right|
\leq 
\lim_{k\rightarrow\infty}s^{-1}\left(\int_{\Omega}\mod{U^{1/2}_{n_{k}}-U^{1/2}}^{2}|\varphi|^{2}\right)^{1/2}=0.
$$
It follows that
$$
\int_{\Omega}\widetilde{w}\varphi=\int_{\Omega}U^{1/2}u\varphi,\quad \forall \varphi\in C^{\infty}_{c}(\Omega),
$$
so $\widetilde{w}=U^{1/2}u$. 

We proved that in $L^2$ we have
\begin{equation}\label{eq: 9}
u_{n}\rightarrow u,\quad \nabla u_{n_{k}}\rightharpoonup \nabla u\quad \text{and}\quad U^{1/2}_{n_{k}}u_{n_{k}}\rightharpoonup U^{1/2}u.
\end{equation}
We now show that the last two convergences in \eqref{eq: 9} are in the normed topology of $L^{2}(\Omega)$.

By ellipticity,
$$
\aligned
J_{n_{k}}& := s\|u_{n_{k}}-u\|^{2}_{2}+\lambda\|\nabla u_{n_{k}}-\nabla u\|^{2}_{2}+\|U^{1/2}_{n_{k}}u_{n_{k}}-U^{1/2}u\|^{2}_{2}\\
&\leq s\|u_{n_{k}}\|^{2}_{2}+s\|u\|^{2}_{2}-2s\Re\int_{\Omega}u_{n_{k}}\bar u+\|U^{1/2}_{n_{k}}u_{n_{k}}\|^{2}_{2}+\|U^{1/2}u\|^{2}_{2}\\
&\hskip 12pt-2\Re\int_{\Omega}U^{1/2}_{n_{k}}U^{1/2}u_{n_{k}}\bar u+\Re\int_{\Omega}\sk{A(\nabla u_{n_{k}}-\nabla u)}{(\nabla u_{n_{k}}-\nabla u)}\\
&={I^{0}}+{I^{1}_{n_{k}}}+{I^{2}_{n_{k}}}+{I^{3}_{n_{k}}},
\endaligned
$$
where
$$
\aligned
&{I^{0}}=s\|u\|^{2}_{2}+\Re\int_{\Omega}\sk{A\nabla u}{\nabla u}+\|U^{1/2}u\|^{2}_{2},\\
&{I^{1}_{n_{k}}}=s\|u_{n_{k}}\|^{2}_{2}+\Re\int_{\Omega}\sk{A\nabla u_{n_{k}}}{\nabla u_{n_{k}}}+\|U^{1/2}_{n_{k}}u_{n_{k}}\|^{2}_{2},\\
&{I^{2}_{n_{k}}}=-2s\Re\int_{\Omega}u_{n_{k}}\bar u-2\Re\int_{\Omega}U^{1/2}_{n_{k}}U^{1/2}u_{n_{k}}\bar u,\\
&{I^{3}_{n_{k}}}=-\Re\left(\int_{\Omega}\sk{A\nabla u_{n_{k}}}{\nabla u}+\int_{\Omega}\sk{A\nabla u}{\nabla u_{n_{k}}}\right).
\endaligned
$$
Sending $k\rightarrow\infty$, we obtain
\begin{itemize}
\item
${\displaystyle 
{I^{0}}=\Re\int_{\Omega}\left(\left(s+L_{U}\right)u\right)\bar u=\Re\int_{\Omega}f\bar u,}$
because $u\in\cD(L_{U})$;
\item
${\displaystyle 
{I^{1}_{n_{k}}}=\Re\int_{\Omega}\left(\left(s+L_{U_{n_{k}}}\right)u_{n_{k}}\right)\bar u_{n_{k}}=\Re\int_{\Omega}f\bar u_{n_{k}}\rightarrow\Re\int_{\Omega}f\bar u,}$

because 
$u_{n_{k}}\rightarrow u$ in $L^{2}(\Omega)$;
\item
${\displaystyle 
{I^{2}_{n_{k}}}=-2\Re\left(s\int_{\Omega}u_{n_{k}}\bar u+\int_{\Omega}\left(U^{1/2}_{n_{k}}u_{n_{k}}\right)\left(U^{1/2}\bar u\right)\right)
\rightarrow-2s\|u\|^{2}_{2}-2\|U^{1/2}u\|^{2}_{2},}$

by \eqref{eq: 9}, since $u\in\cD(\a_{U})$ implies $U^{1/2}\bar u\in L^2(\Omega)$; and finally
\item
${\displaystyle 
{I^{3}_{n_{k}}}\rightarrow-2\Re\int_{\Omega}\sk{A\nabla u}{\nabla u}
}$

by \eqref{eq: 9} again, since $A\in\cA(\Omega)$ implies $|A\nabla u|,|A^*\nabla u|\leqsim |\nabla u|\in L^2(\Omega)$.
\end{itemize}
Therefore, using that  $u\in\cD(L_{U})$, we obtain, as $k\rightarrow\infty$,
$$
\aligned
I^{0}+I^{1}_{n_{k}}   
& \rightarrow 2\Re\int_{\Omega}f\bar u\\
{I^{2}_{n_{k}}}+{I^{3}_{n_{k}}}
&\rightarrow-2\Re\int_{\Omega}\left(\left(s+L_{U}\right)u\right)\bar u=-2\Re\int_{\Omega}f\bar u.
\endaligned
$$
It follows that $J_{n_{k}}\rightarrow 0$ as $k\rightarrow\infty$, so
\begin{equation}\label{eq: 10}
\nabla u_{n_{k}}\rightarrow \nabla u\quad \text{and}\quad U^{1/2}_{n_{k}}u_{n_{k}}\rightarrow U^{1/2}u
\hskip 30pt
\text{in }L^2(\Omega), 
\end{equation}
as desired.

By repeating verbatim the argument following  \eqref{eq: 9}, we may prove that every {\it subsequence} of $ (u_{n})_{n}$ has its own subsequence for which \eqref{eq: 10} holds.
Therefore, by a standard convergence argument involving subsequences, \eqref{eq: a1a} and \eqref{eq: a2a} hold for all $\zeta=-s$, $s>0$. (Recall that for \eqref{eq: a11a} we already know that, by virtue of \eqref{eq: 6}.)
\smallskip

It remains to prove \eqref{eq: a11a}, \eqref{eq: a1a} and \eqref{eq: a2a} for all $\zeta\in({\mathbb C}\setminus\overline\bS_{\nu })\setminus(-\infty,0)$. 
Fix $f\in L^{2}(\Omega)$ and for each $n\in\N\cup\{\infty\}$ consider the function 
$$
G_{n}:{\mathbb C}\setminus\overline\bS_{\nu }\longrightarrow L^{2}(\Omega)\times L^{2}(\Omega)\times L^{2}(\Omega;{\mathbb C}^{d})=:H
$$ 
given by
$$
G_{n}(\zeta)
=\left(u_n(\zeta),U^{1/2}_{n}u_n(\zeta),\nabla u_n(\zeta)\right).
$$ 
The functions $G_{n}$ are holomorphic, 
because the complex-valued function $\sk{G_{n}(\cdot)}{\bf g}_{H}$ is holomorphic for all ${\bf g}$ in the norming subspace $L^{2}(\Omega)\times C^{\infty}_{c}(\Omega;{\mathbb C}^{d})\times C^{\infty}_{c}(\Omega)$ of $H$, cf. \cite[Proposition A.3]{ABHbook}.
Owing to \eqref{eq: 3bis}, the family $\{G_{n}: n\in\N_{+}\}$ is locally uniformly bounded in ${\mathbb C}\setminus\overline\bS_{\nu }$. 
Therefore, since we have already proved that $G_{n}(-s)\rightarrow G_{\infty}(-s)$ in $H$ for all $s>0$, it follows from Vitali's theorem \cite[Theorem~A.5]{ABHbook} that the convergence holds true for all $\zeta\in{\mathbb C}\setminus\overline\bS_{\nu }$.
\end{proof}

\subsubsection{Proof of Theorem \ref{t: 1a}}
\label{s: Teske boje}
We use the
standard representation of 
the analytic semigroup $T^{U_{n}}_{t}$, $n\in\N\cup\{\infty\}$, by means of a Cauchy integral (see, for example, \cite[Chapter~II]{EN} or \cite[Lemma 2.3.2]{Haase}). We used it earlier in the proof of \cite[Lemma A.4]{CD-DivForm}.
Fix $\delta>0$, $\theta>\nu(A)$ and denote by $\gamma$ the positively oriented boundary of $\bS_\theta\cup\mn{\zeta\in{\mathbb C}}{|\zeta|<\delta}$. Then 
\begin{equation*}
\aligned
\label{eq: Vrijeme je da se krene}
\nor{\nabla T^{U_{n}}_{t}f - \nabla T^{U}_{t}f}_2
& \leqsim
\int_\gamma
e^{-t\,\Re\zeta}
\nor{
\nabla\left(\zeta-L_{U_{n}}\right)^{-1}f 
-\nabla\left(\zeta-L_{U}\right)^{-1}f 
}_2
\,|d\zeta|,
\\
\nor{U_n^{1/2}T^{U_{n}}_{t}f - U^{1/2}T^{U}_{t}f}_2
& \leqsim
\int_\gamma
e^{-t\,\Re\zeta}
\nor{
U_n^{1/2}\left(\zeta-L_{U_{n}}\right)^{-1}f 
-
U^{1/2}\left(\zeta-L_{U}\right)^{-1}f 
}_2
\,|d\zeta|.
\endaligned
\end{equation*}
By Proposition~\ref{p: 1a}, 
the integrands converge to zero as $n\rightarrow\infty$. 
An examination of the constant $C_{\lambda,\nu}(\zeta)$ from \eqref{eq: 3bis} shows that for $\zeta$ along the curve $\gamma$ we have $C_{\lambda,\nu}(\zeta)\leqsim 1$ uniformly in $n\in\N\cup\{\infty\}$ and $\zeta\in\gamma$.
This means that we can apply the dominated convergence theorem in the two integrals above and complete the proof.
\qed

\section*{Appendix: Invariance of form-domains under normal contractions}

\renewcommand{\thesection}{\Alph{section}}
\setcounter{section}{1}
\setcounter{thm}{0}
Following \cite[Section 2.4]{O}, we say that a function $p:{\mathbb C}\rightarrow{\mathbb C}$ is a {\it normal contraction} if it is Lipschitz on ${\mathbb C}$ with constant one and $p(0)=0$. Denote by $\oN$ the set of all normal contractions. Define $T:{\mathbb C}\rightarrow{\mathbb C}$ by 
$
 T(\zeta)=(\Re\zeta)^+, 
$
 where $x^+=\max\{x,0\}$. 
Recall that we defined $P:{\mathbb C}\rightarrow{\mathbb C}$ in \eqref{eq: P}. 
Functions $P,T$ belong to the class $\oN$. Moreover, they are in a particular sense fundamental representatives of this class, as we show next. 

\begin{prop}
Let $\oV$ be a closed subspace of $H^{1}(\Omega)$ containing $H^{1}_{0}(\Omega)$. Then $\oV$ is invariant under $P$ and $T$ if and only if it is invariant under the (whole) class $\oN$.
\end{prop}

\begin{proof}
Suppose that $\oV$ is invariant under $P$ and $T$. Let $\Delta$ be the euclidean Laplacian on $\Omega$ subject to the boundary conditions embodied in $\oV$. That is, $-\Delta$ is the operator arising from the form $\gotb$, defined by $\cD(\gotb)=\oV$ and 
\begin{equation}
\label{eq: kocka}
\gotb(u,v)=\int_\Omega\sk{\nabla u}{\nabla v}_{{\mathbb C}^{d}}.
\hskip 30pt
\forall\, u,v\in\oV.
\end{equation}
Parts {\it 1)} and {\it 2)} of \cite[Theorem 4.31]{O} imply that $\left(e^{-t(-\Delta)}\right)_{t>0}$ is sub-Markovian. Now \cite[Theorem 2.25]{O} implies that $\oV$ is invariant under $\oN$.

The implication in the opposite direction is obvious, as $P,T\in\oN$.
\end{proof}

\begin{prop}
When $\oV$ is any of the special cases (a)-(c) from Section \ref{s: Karlstejn}, then $\oV$ is invariant under $\oN$.
\end{prop}

\begin{proof}
By \cite[Theorem 2.25]{O}, it suffices to find a sesqulinear form $\gotb$ such that:
\begin{itemize}
\item
$\cD(\gotb)=\oV$
\item
$\gotb$ is symmetric, accretive and closed on $L^2(\Omega)$;
\item
if $\oL$ is the operator associated with $\gotb$ then the semigroup $\exp(-t\oL)$ is sub-Markovian.
\end{itemize}
We define $\gotb$ on $\oV$ by \eqref{eq: kocka}. Thus $\oL=-\Delta$ with boundary conditions embodied in $\oV$. 
Then $\cD(\gotb)=\oV$ by construction, the form is clearly symmetric and accretive. It is closed by Theorem~\ref{t: majke cigan}. In order to check that the semigroup is sub-Markovian, we have to check, cf. \cite[Definition 2.12]{O}, that it is positive and contractive on $L^\infty$. Now, these properties are proven in \cite[Corollary 4.3]{O} and \cite[Corollary 4.10]{O}, respectively.
\end{proof}

\section*{Acknowledgements}

A. Carbonaro was partially supported by the ``National Group for Mathematical Analysis, Probability and their Applications'' (GNAMPA-INdAM).

O. Dragi\v{c}evi\'c was partially supported by the Slovenian Research Agency, ARRS (research grant J1-1690 and research program P1-0291).

\bibliography{biblio_Potentials}{}
    \bibliographystyle{alpha}

\end{document}